\documentclass[11pt]{article}
\usepackage[dvips,a4paper,hmargin={3.5cm,3.5cm},vmargin={3.2cm,3.2cm}]{geometry}

\mathcode`\<="213C
\mathcode`\=="203D
\mathcode`\>="213E
\let\oldle\le\def\le{\mathbin{\oldle}} 
\usepackage{pstricks}
\usepackage{comment}
\usepackage{setspace}
\usepackage{amsthm}
\usepackage[alphabetic]{amsrefs}
\usepackage{amssymb}
\usepackage{dsfont}
\usepackage{mathtools}
\usepackage{yfonts}
\usepackage{euscript}
\usepackage{MnSymbol}
\usepackage{hyperref}
\usepackage{ulem}
\usepackage{auto-pst-pdf}
\usepackage[hang,flushmargin]{footmisc} 

\renewenvironment{itemize}
  {\nobreak\begin{list}{$\triangleright$}{%
   \setlength{\parskip}{0mm}
   \setlength{\topsep}{.5\baselineskip}
   \setlength{\rightmargin}{0mm}
   \setlength{\listparindent}{0mm}
   \setlength{\itemindent}{0mm}
   \setlength{\labelwidth}{3ex}
   \setlength{\itemsep}{.5\baselineskip}
   \setlength{\parsep}{0mm}
   \setlength{\partopsep}{0mm}
   \setlength{\labelsep}{1ex}
   \setlength{\leftmargin}{\labelwidth+\labelsep}
   }}{%
   \end{list}\vspace*{-.5\baselineskip}}

\def\models{\vDash}
\def\proves{\vdash}
\def\ZZ{\mathds Z}
\def\NN{\mathds N}
\def\QQ{\mathds Q}
\def\RR{\mathds R}

\def\DD{\mathds D}

\def\IMP{\Rightarrow}

\def\sm{\smallsetminus}

\def\notmodels{\parbox{2.5ex}{\hfil$\vDash$\llap{\raisebox{.0ex}{{\small$/\;$}}}}}
\def\notproves{\parbox{2.5ex}{\hfil$\vdash$\llap{\raisebox{.1ex}{{\;\small$/$}}}}}

\def\Th{{\rm Th}}

\def\<{\langle}
\def\>{\rangle}
\def\0{\varnothing}
\def\theta{\vartheta}
\def\phi{\varphi}
\def\epsilon{\varepsilon}
\newrgbcolor{blue}{0 0 .7}
\def\emph#1{{\blue\boldmath\bfseries #1}}
\def\ssf#1{\textsf{\small #1}}

\newtheoremstyle{mio}
     {2\parskip}
     {0mm}
     {\sl}
     {}
     {\bfseries}
     {}
     {3mm}
     {\llap{\thmnumber{#2}\hskip2mm}\thmname{#1}\thmnote{\bfseries{}#3}}

\newcounter{thm} 

\theoremstyle{mio}

\newtheorem{proposition}[thm]{Proposition}
\newtheorem{void}[thm]{}
\newtheorem{lemma}[thm]{Lemma}

\newtheoremstyle{plain}
     {2\parskip}
     {0mm}
     {}
     {}
     {\bfseries}
     {}
     {3mm}
     {\llap{\thmnumber{#2}\hskip2mm}\thmname{#1}\thmnote{\bfseries{} #3}}

\theoremstyle{plain}
\newtheorem{example}[thm]{Example}
\newtheorem{remark}[thm]{Remark}

\def\lle{\le}

\def\leq{=}

\title{The $\RR$eal truth}
\author{Stefano Baratella \and Domenico Zambella}

\begin{document}
\date{}
\maketitle

\begin{abstract}\footnotetext{
 Unbounded  signed truth values. \L ukasiewicz logic.  }
 
\noindent We study a real valued propositional logic with unbounded positive and negative truth values that we call \textit{$\RR$-valued logic}. Such a logic is semantically equivalent to continuous propositional logic, with a different choice of connectives. After presenting the deduction machinery and the semantics of $\RR$-valued logic, we prove a  completeness theorem for finite theories. Then we define unital and Archimedean theories, in accordance with the theory of Riesz spaces. In the unital setting,  we prove the equivalence of consistency and satisfiability and an approximated completeness theorem similar to  the one that holds for continuous propositional logic.  Eventually, among unital theories, we characterize Archimedean theories as those for which strong completeness holds. We also point out that $\RR$-valued logic provides alternative calculi for \L ukasiewicz logic and for propositional continuous logic.
\end{abstract}

\section{Introduction}
Many-valued logics date back to the very early development of mathematical logic. The real unit interval $[0,1]$ has in particular been  a favorite set of truth values. Among $[0,1]$-valued logics, \textit{\L ukasiewicz logic} occupies a predominant place (see e.g.\@ \cite{Hajek}). Initially motivated by philosophical considerations and scientific curiosity, \L ukasiewicz logic has received considerable attention by philosophers and computer scientists working on fuzzy logics, approximate reasoning  or other forms of non-classical reasoning. On the algebraic side, the study of \L ukasiewicz logic led to the notion  of an MV-algebra. C.C. Chang first showed that MV-algebras provide a complete semantics for \L ukasiewicz logic. See \cite{Chang1},~\cite{Chang2}. 

\textit{Continuous logic} makes its first appearence in \cite{BYU}. Its propositional fragment extends  \L ukasiewicz logic. See, for instance, the overview in \cite{BYP}. In spite of appearing just an  extension of \L ukasiewicz logic, continuous logic has an independent origin and  different motivations. Actually, it builds on a field of research, initiated in the 1980's by Henson and continued by a number of authors, on  the model theory of so called \textit{metric structures}.  Continuous logic provides a suitable logical setting for dealing with structures like Banach spaces, Banach algebras, probability algebras, etc. that arise in functional analysis or probability theory. See \cite{BYBHU}. 

As the origin of continuous logic lies in model theory, issues like developing deduction systems and studying their completeness have been somewhat postponed. A completeness theorem for continuous predicate logic appears in \cite{BYP}. The authors make use of results from~\cite{BYrv}, where a completeness result for propositional continuous logic is derived from the corresponding result for \L ukasiewicz logic.

Being $[0,1]$-valued, continuous logic apparently deals with bounded structures only. Such a limitation can be overcome  by allowing many-sorted structures, but the many-sorted approach suffers from some drawbacks (see~\cite{BY}). Those can only be avoided by passing to a genuine logic for unbounded structures: one such logic has been introduced in~\cite{BY}.

Strongly related to our work is an unbounded real valued logic that predates \cite{BY} and has been introduced with different motivations. We refer to  the \textit{abelian logic}  of  \cite{MS} (independently introduced in  \cite{Ca}).    In  \cite{MS}, the authors provide a sound and complete axiomatization of abelian logic with respect to the class of abelian $l$-groups (equivalently: with respect to the $l$-group of the reals). The connections between abelian and \L ukasiewicz logic and calculi for both logics have been studied in \cite{MO}.  Among  other results, in  \cite{MO}  the authors prove  soundness and completeness of a hypersequent calculus and of  labelled and unlabelled sequent  calculi for abelian logic. They also establish the computational complexity of the labelled calculus.
  
In this paper we  develop a syntactic calculus in Hilbert style for a real-valued propositional logic,  named \textit{$\RR$-valued logic}.  Our logic can be viewed as an extension of abelian logic. As already mentioned, abelian logic is the logic of lattice-ordered abelian groups, which has the reals  among its characteristic models. 

Ours  is the logic of lattice-ordered vector spaces, also known as Riesz spaces (see \cite{Aliprantis}).  As proved in Section~\ref{completeness}, $\RR$-valued logic  has the reals among its characteristic models (see, in particular,   Remark~\ref{compl_char_mod} below). From the syntactic viewpoint, $\RR$-valued logic comes equipped with scalar multiplication as an additional formation rule for formulas. In this paper we restrict  scalar multiplication to multiplication by rationals in order not to rule the possibility of working with a countable language (this might be useful in future developments). The results obtained in this paper extend to multiplication by reals in a straightforward manner.

There is a natural overlapping of the results proved in this paper with those obtained in \cite{MS}, \cite{MO} or, more in general, in the area of  \textit{mathematical fuzzy logic}  (see  \cite{AA.VV.1} and  \cite{AA.VV.2} for a comprehensive and up-to-date  presentation of the latter).  Nevertheless the proof techniques used in the different contexts have their own  distinctive vein: algebraic (in this paper); logically-minded (in \cite{MS}) and proof-theoretic (in \cite{MO}).  The different contexts in which abelian and $\RR$-valued logic are developed may make a step-by-step translation of one logic into the other not completely straightforward. We do not investigate that issue because we believe that the strong connections between the two logics are witnessed, beyond any doubt, by their respective completeness theorems.

Yet we may say  that some of the completeness results  proved  in this paper can be regarded as extensions of  those in \cite{MS} or in \cite{MO}, as they apply to a  logic more expressive than abelian logic.

A further reason why, in our opinion,  $\RR$-valued logic is worth being investigated  is that it creates a link between the two independently developed areas of continuous logic and  of  mathematical fuzzy logic. 

We shall investigate a predicate extension of  $\RR$-valued logic in a future work. As already mentioned, a syntactic calculus for predicate continuous logic  is presented in \cite{BYP}, while the semantics of unbounded continuous logic is introduced  in \cite{BY}. From the latter it is evident that there is no straightforward interpretation of  the quantifiers in the unbounded case. Difficulties are also met on the abelian logic side. Predicate extensions of abelian logic are  very quickly sketched in \cite[\S X]{MS}, where the authors admit their  conflicting ideas about the treatment of quantifiers. In any case, ours remains  a necessary step towards a possible predicate extension.

In the framework of $\RR$-valued logic we introduce a  notion of \textit{theory}  and, under suitable assumptions on theories, we prove different formulations of completeness  with respect to a semantics that we introduce in Section~\ref{structures}. 

In Section~\ref{axioms} we introduce the syntactic calculus of $\RR$-valued logic, which turns out to be quite different from the one introduced in \cite{BYP} and~\cite{BYrv}.  Having in mind the axiomatization of Riesz spaces (i.e. of  lattice-ordered vector spaces,  see \cite{Aliprantis}), our logical axioms turn out to be quite natural. As for deduction rules, we show that modus ponens alone does not suffice to get completeness.  For this reason we have to introduce two additional rules. Then we prove a number of properties of our provability relation, including the correctness of a sort of cut rule, see Proposition~\ref{prop_cut_elim}.

It is a standard terminology to say that, in  a given logic, \textit{strong completeness} holds for a set $\Sigma$ of formulas if $\Sigma\models\phi$ if and only if $\Sigma\proves\phi$, for every formula $\phi.$ If the previous property holds for all (finite) $\Sigma$, one says that \textit{(finite) strong completeness} holds for that logic. We recall that finite strong completeness holds for \L ukasiewicz logic and for continuous propositional logic.
 
In Section~\ref{lin_for} we prove finite strong completeness of  $\RR$-valued logic with respect to a suitable subclass of formulas. The proof is obtained by reducing finite strong completeness to a problem in convex analysis, whose solution is provided by suitable formulations of Farkas' Lemma.

In Section~\ref{completeness} we  extend   the  completeness theorem obtained in Section~\ref{lin_for} to all formulas.  At the end of Section~\ref{completeness} we also show that $\RR$-valued logic provides alternative calculi for \L ukasiewicz logic and for continuous propositional logic.

Another weak formulation of completeness, namely the equivalence of consistency and satisfiability, holds for \L ukasiewicz logic and for continuous propositional logic. We refer to the equivalence of consistency and satisfiability as  \textit{weak} completeness. Continuous propositional logic also satisfies a further weak version formulation of completeness, a so called \textit{approximated} completeness (see \cite{BYP}). We  show that validity of corresponding results in $\RR$-valued logic is not granted without an additional  hypothesis, as  not every consistent set of assumptions extends to a maximal consistent one.

In Section~\ref{Archimedean} we define the classes of unital and Archimedean theories. We show that weak and approximated completeness  (in the sense of continuous logic) both hold for unital theories. Finally, among unital theories, we characterize Archimedean theories as those for which strong completeness holds.

\section{Formulas, structures and theories}\label{structures}

In this section we begin the description of $\RR$-valued logic. We remain rather informal, in order not to bother a supposedly experienced reader with a number of minor details.

We identify  a \emph{language} $L$ with  its extralogical symbols. So language $L$ is a set of  symbols that we call  \emph{proposition letters}.  The set of \emph{formulas} is the least set containing the proposition letters and closed under the connectives in the set $\{0, +,\wedge\}\cup\QQ$, where $0$ is a logical constant; rational numbers are unary connectives; $+, \wedge$ are binary connectives. 

Our choice of connectives is inspired by Riesz spaces (see \cite{Aliprantis}), in the same way as  the connectives of classical logic are inspired by Boolean algebras.

We also consider an extension of logical symbols  obtained by adding  the logical constant $1$. We shall refer to these two settings as the \emph{restricted} and the \emph{extended} case respectively.
 
More explicitly, all proposition letters are formulas and, if $\phi$ and $\psi$ are formulas, then so are:

\begin{minipage}[t]{.599\textwidth}
\begin{itemize}
\item[0.] $0$ and,  in the extended case, $1$;
\item[1.] $\psi+\xi$;
\end{itemize}
\end{minipage}
\begin{minipage}[t]{.399\textwidth}
\begin{itemize}
\item[2.] $\psi\wedge\xi$;
\item[3.] $q\phi$  for every $q\in\QQ$.\end{itemize}

\end{minipage}

\medskip
Actually, we have been informal above as, for instance, braces will soon appear.   We write $-\phi$ for $(-1)\phi$ and $\phi-\psi$ for $\phi+(-\psi)$. We write $\phi\vee\psi$ for $-(-\phi\wedge\nobreak-\psi)$. We also borrow some notation from Riesz spaces: $\phi^+$ and $\phi^-$ stand for $0\vee\phi$ and  $0\vee(-\phi)$ respectively. We write $|\phi|$ for $\phi\vee(-\phi)$. In the extended case, we abbreviate $q1$ with $q$, for all $q\in\QQ$. 
 
A \emph{structure}, or \emph{model}, for  a language $L$ is a function $M:L\to \RR$. In the extended case we further require that $M$ takes values in the interval $[-1,1]$. This semantic requirement and its axiomatic counterpart (axiom \ssf{a15} below) are motivated by the consequences of Proposition~\ref{prop_ext_bound} below.

Let $M$ be a model. We recursively define  $\phi^M$ for an arbitrary formula  $\varphi$  as follows:
\begin{minipage}[t]{.599\textwidth}
\begin{itemize}
\item[0.] $0^M=0$;
\item[1.] $1^M=1$, in the extended case;
\item[2.] $P^M=M(P)$, for all proposition letters $P$; 
\end{itemize}
\end{minipage}
\begin{minipage}[t]{.399\textwidth}
\begin{itemize}
\item[3.] $(\psi+\xi)^M=\psi^M+\xi^M$;
\item[4.] $(\psi\wedge\xi)^M=\min\big\{\psi^M,\xi^M\big\}$;
\item[5.] $(q\psi)^M=q\psi^M$. 
\end{itemize}
\end{minipage}
\medskip

We use the word \emph{inequality\/} as a synonym for an \emph{ordered pair of formulas}. Inequalities are denoted by writing $\phi\lle\psi$. In order to simplify notation, we abbreviate the inequality $0\le\phi$ with $\phi$. 

A \emph{theory\/} is a binary relation on the set of formulas, namely a set of inequalities. We elaborate on this definition of theory in Remark~\ref{rk_solid} below.

If $M$ is a model and $\phi^M\le\psi^M$ we write $M\models\phi\lle\psi$ and we say that $\varphi\le\psi$ holds in $M$. If $T$ is a theory, we define $M\models T$ and $T\models\phi\lle\psi$ in the usual way. We write $M\models 0<\phi$ if $0<\phi^M$. Finally, we write $\Th(M)$ for the set of inequalities that hold in $M$.

We finish  this section by commenting on our choice of logical connectives. The presence of a unary connective for each element of $\QQ$ is just a matter of convenience. Alternative meaningful choices are $\{-1\}$ or $\{-\frac12\}$. Together with addition, the former singleton yields a connective for each element of $\ZZ$. The latter yields a connective for each element of the set $\DD$ of dyadic rationals. With respect to both choices, logical axioms  for lattice modules replace  those for vector lattices that occur in our setting. Having in mind an extension to the predicate case, we opt for a complete (in the sense of \cite{BYU}) set of connectives. In this regard, $\DD$ and $\QQ$ are equivalent choices.

\section{Logical axioms and derivations}\label{axioms}
The following inequalities, where $\phi$, $\psi$ and $\xi$ range over all formulas, are called \emph{logical axioms\/}. We write $\phi=\psi$ to denote the theory $\{\phi\le\psi, \psi\le\phi\}$. 

So an axiom expressing an equality (see below) actually stands for  a pair of axioms. We also write $\phi\le\xi\le\psi$ to denote the theory $\{\phi\le\xi, \xi\le\psi\}.$ The reader can make sense by her-/him-self of some other minor notational abuses.

In the following, $\QQ^+$ denotes the set of nonnegative rationals.
 
There are two axiom groups. The axioms from the first group are chosen having in mind the theory of vector spaces over $\QQ$:

\begin{minipage}[t]{.499\textwidth}
\begin{itemize}
\def\lle{\ \le\ }
\def\leq{\ =\ }
\item[a1.] $\phi+\psi\leq\psi+\phi$
\item[a2.] $(\phi+\psi)+\xi\leq\psi+(\phi+\xi)$
\item[a3.] $\phi+0\leq\phi$
\item[a4.] $1\,\phi\leq\phi$
\end{itemize}
\end{minipage}
\begin{minipage}[t]{.499\textwidth}
\begin{itemize}
\def\lle{\ \le\ }
\def\leq{\ =\ }
\item[a5.] $0\,\phi\leq0$
\item[a6.] $r\,\phi+s\,\phi\leq(s+r)\,\phi$
\item[a7.] $r\,\phi+r\,\psi\leq r\,(\phi+\psi)$
\item[a8.] $r\,(s\,\phi)\leq(rs)\,\phi$
\end{itemize}
\end{minipage}
\medskip

The axioms from the second group are inspired by the theory of Riesz spaces:

\begin{minipage}[t]{.499\textwidth}
\begin{itemize}
\def\lle{\ \le\ }
\def\leq{\ =\ }
\item[a9] $\phi\wedge\phi\leq\phi$
\item[a10.] $\phi\wedge\psi\leq\psi\wedge\phi$
\item[a11.] $(\phi\wedge\psi)\wedge\xi\leq\phi\wedge(\psi\wedge\xi)$
\end{itemize}
\end{minipage}
\begin{minipage}[t]{.499\textwidth}
\begin{itemize}
\def\lle{\ \le\ }
\def\leq{\ =\ }
\item[a12.] $(\phi+\xi)\wedge(\psi+\xi)\leq\phi\wedge\psi + \xi$
\item[a13.] $r(\phi\wedge\psi)\leq r\phi\wedge r\psi$ for $r\in\QQ^+$
\item[a14.] $\phi\wedge\psi\lle\psi$
\end{itemize}
\end{minipage}
\smallskip

 In the extended case, we  add the following axioms for every proposition letter $P$:\nobreak
\begin{itemize}
\item[a15] $-1\lle P \lle 1$.
\end{itemize}

\def\provesL{\,\vdash\hskip-1.7ex\raisebox{-.5ex}{\tiny\sf lin}\ }
\def\provesMP{\,\vdash\hskip-1.7ex\raisebox{-.5ex}{\tiny\sf mp}\ }

There are three inference rules that are listed below. In the sequel we shall provide some arguments in favour of their mutual independence and their non-replaceability with logical axioms, even though we are not primarily concerned with these issues.

\begin{itemize}
\item[r1.]  $\phi\lle\xi\lle\psi\ \proves\ \phi\lle\psi$;\hfill (Transitivity or modus ponens)

\item[r2.] $\phi\lle\psi\  \proves\  r\phi+\xi\ \lle\  r\psi+\xi$ 
for $r\in\QQ^+$;\hfill (Positive linearity)

\item[r3.] $\phi\lle\psi\ \proves\  \phi\wedge0\ \lle\  \psi\wedge0$.  \hfill (Restriction)

\end{itemize}
\def\provesi{\,\vdash\hskip-1.4ex\raisebox{-.5ex}{\tiny\sf i}\ }
The notion of derivation is the standard one in Hilbert systems. As is customary, we write $T\vdash \varphi\le\psi$  if there exists a derivation  of  $\varphi\le\psi$ from $T$.
We write $\provesMP$ for derivability from \ssf{r1} only and $\provesL$ for derivability from rules \ssf{r1} and \ssf{r2} only.

The following proposition states some facts that we shall frequently use in the sequel. The first fact states the invertibility of rule \ssf{r1};  the second one is a generalization of  \ssf{r3}.

\begin{proposition}\label{prop_basic_inferences1}
The following hold for all formulas $\phi, \xi, \psi$:\\
\begin{minipage}[t]{.649\textwidth}
\begin{itemize}
\item[1.] $r\phi + \xi\;\lle\;r\psi + \xi\ \provesL\  \phi\lle\psi$ for every $0<r\in\QQ$;
\item[2.] $\phi\lle\psi\ \proves\  \phi\wedge\xi\ \lle\ \psi\wedge \xi$;
\end{itemize}
\end{minipage}
\begin{minipage}[t]{.349\textwidth}
\begin{itemize}
\item[3.] $\xi\lle\phi,\ \xi\lle\psi\ \proves\  \xi\ \lle\ \phi\wedge\psi$;
\item[4.] $\phi\lle\xi,\ \psi\lle\xi\ \proves\  \phi\vee\psi\ \lle\ \xi$.
\end{itemize}
\end{minipage}
\end{proposition}

\begin{proof} (Sketch)
\begin{itemize}

\item[1.] Get $r^{-1}r\phi + \xi-\xi\lle r^{-1}r\psi + \xi-\xi$ by \ssf{r2}. Another application of \ssf{r2} to axioms,  together  with \ssf{r1}, yields $\phi\lle\psi$.

\item[2.] Get $\phi-\xi\lle\psi-\xi$ by \ssf{r2} and $(\phi-\xi)\wedge0\lle(\psi-\xi)\wedge0$ by \ssf{r3}. Add $\xi$ on both sides and apply \ssf{a12}.

\item[3.] Get $\xi\wedge\psi\lle\phi\wedge \psi$ from  $\xi\lle\phi$ by \ssf{2} and $\xi\lle\xi\wedge\psi$ from  $\xi\lle\psi$, then apply \ssf{r1}.

\item[4.] Follows from \ssf{3}, since $\phi\lle\psi\proves -\psi\lle-\phi$.\qedhere
\end{itemize}
\end{proof}

\begin{remark} Notice that, for every theory $T$ and every inequality $\phi\le\psi,$ $$T\proves \phi\le\psi\,\Leftrightarrow\, T\proves \psi-\phi.$$ 

Implication $\Rightarrow$ follows by application \ssf{r2}.  The converse implication is a consequence of \ssf{1}, Proposition~\ref{prop_basic_inferences1}.  The same equivalence holds for $\provesL$.  We shall use these facts without further mention in what follows (in particular in the formulation of  the soundness and completeness theorems).

\end{remark}

The following is straightforward:

\begin{void}[Soundness Theorem]\label{prop_soundness}
For every theory $T$ and every formula  $\varphi$, if $T\proves\phi$ then $T\models\phi$.
\end{void}

We shall prove in the sequel that the converse implication holds for finite theories and, under additional assumptions, for infinite theories as well.

A na\"{\i}ve formulation of the classical deduction theorem for linear derivability would  be the following: if $T,\theta\provesL\psi$ then $T\provesL\theta\lle\psi$. Unfortunately this does not hold. For it holds that $2Q\lle P, Q \provesL P$,  but $2Q\lle P \not\provesL Q\lle P$, by Theorem~\ref{prop_soundness}. The same argument  applies to $\proves$ as well. The previous counterexample also appears in \cite{BYP}. Nevertheless weaker formulations of the deduction theorem hold for both deductive systems. Even if we are not  going to establish  a completeness theorem with respect to $\provesL,$   we deal first with it in order to make the reader acquainted with the deductive system.

\begin{void}[Linear Deduction Theorem]\label{prop_lin_ded}
The following are equivalent:
\begin{itemize}
\item[1.] $T,\,\theta \provesL\phi\lle\psi$;
\item[2.] $T\provesL\phi+r\theta\lle\psi$ for some $r\in\QQ^+$.
\end{itemize}
\end{void}

\begin{proof}
As \ssf{2}$\IMP$\ssf{1} is trivial, we prove \ssf{1}$\IMP$\ssf{2}. We argue by induction on the length of a derivation of $\varphi \le \psi$ from $T, \theta$.  If the length is $1$, then either $\phi\lle\psi$ is in $T$ or it is an axiom, in which case we take $r=0$, or $\phi\lle\psi$ is syntactically equal to $0\lle\theta$, so the conclusion follows by taking $r=1$.  

If the last rule applied in the derivation is \ssf{r1} then $T,\theta\provesL\phi\lle\zeta$ and $T,\theta \provesL\zeta\lle\psi$,  for some formula $\zeta$. By induction hypothesis, $T\provesL\phi+r_1\theta\lle\zeta$ and  $T\provesL\zeta+r_2\theta\lle\psi$, for some $r_1, r_2\in \QQ^+$. The conclusion follows by taking  $r=r_1+r_2$.

If the last rule applied in the derivation is \ssf{r2} then there exist $\phi^\prime, \psi^\prime, \xi$ and $s\in \QQ^+$  such that  $\phi$  and $\psi$ are  $s\phi^\prime+\xi$ and $s\psi^\prime+\xi$ respectively and $T,\theta \provesL\phi^\prime\lle\psi^\prime$. By induction hypothesis $T\provesL\phi^\prime+t\theta\lle\psi^\prime$, for some $t\in \QQ^+$. So the conclusion follows by taking  $r=st$.
\end{proof}

The proof of the Linear Deduction Theorem can be easily adapted to prove a similar statement for $\provesMP$ in place of $\provesL$,  also getting the  stronger conclusion that $r\in \NN$. We exploit this fact to argue that $\provesL$ is indeed stronger than $\provesMP$. Notice that $2P\provesL P$. On the other hand, if it were that  $2P\provesMP P$ then, by the Deduction Theorem for $\provesMP$,  we would get $\provesMP r2P\le P$ for some $r\in\NN$. But the latter is not a valid derivation.

We justify the  presence of rule \ssf{r3} in a similar way. We notice that $P\not\provesL P\wedge0$. Otherwise, by the Linear Deduction Theorem, there would be $r\in\QQ^+$ such that $\provesL rP\le P\wedge0.$ Then $\models rP\le P\wedge0,$ which does not hold.

 \,From the Linear Deduction Theorem we get the following:

\begin{proposition}\label{lin_cut}
The following are equivalent:
\begin{itemize}
\item[1.] $T\provesL\psi$;
\item[2.] $T,\,\phi\provesL \psi$ and $T,-\phi\provesL \psi$.
\end{itemize}
\end{proposition}
\begin{proof}
As \ssf{1}$\IMP$\ssf{2} is trivial, we prove \ssf{2}$\IMP$\ssf{1}. From \ssf{2} and the Linear Deduction Theorem we obtain
\begin{itemize}
\item[] $T\provesL\,r\phi\lle \psi$ and $T\provesL -s\phi\lle \psi$.
\end{itemize}
Assume $r,s>0$, otherwise the conclusion is trivial. Then $T\provesL (r^{-1}+s^{-1})\psi$, from which  $T\provesL\psi$ follows.
\end{proof}

Notice that   \ssf{2}$\IMP$\ssf{1} above states the correctness of a sort of cut rule for  $\provesL$. Below, in Proposition~\ref{prop_cut_elim}, we prove a similar result for $\proves$.

The following are basic identities valid in Riesz spaces (see e.g.~\cite{Aliprantis}) that can also be proved in $\RR$-valued logic. We include a proof sketch for convenience.

\begin{proposition}\label{prop_basic_Riesz}
The following hold for all formulas $\phi, \psi$:\\
\begin{minipage}[t]{.499\textwidth}
\begin{itemize}
\item[1.] $\proves \phi+\psi=\phi\wedge\psi+\phi\vee\psi$;
\item[2.]$\proves \phi=\phi^+-\phi^-$;
\end{itemize}
\end{minipage}
\begin{minipage}[t]{.499\textwidth}
\begin{itemize}
\item[3.]  $\proves \phi^+\wedge\phi^-=0$;
\item[4.]  $\proves |\phi|=\phi^++\phi^-$.
\end{itemize}
\end{minipage}
\end{proposition}
\begin{proof} (Sketch)
\begin{itemize}
\item[1.]  
Get $\proves \psi\lle\phi\vee\psi$ by \ssf{a14}, then $\proves\phi+\psi-\phi\vee\psi\lle\phi$  by \ssf{r2}. Similarly,  from $\proves \phi\lle\phi\vee\psi$ get $\proves\phi+\psi-\phi\vee\psi\lle\psi$.  The $\lle$ inequality in \ssf{1} now follows from \ssf{3} of Proposition~\ref{prop_basic_inferences1}. The opposite inequality can be proved in a similar way, by \ssf{a14} and \ssf{4} of Proposition~\ref{prop_basic_inferences1}.

\item[2] Notice that $\proves\phi^+-\phi^-=\phi\vee0+\phi\wedge0$ and apply \ssf{1}.

\item[3] From \ssf{2} and \ssf{a12} we obtain $\proves\phi^+\wedge\phi^-=(\phi^+-\phi^-)\wedge0+\phi^-=-\phi^-+\phi^-$.

\item[4] Follows from \ssf{2} and \ssf{a12} as  $\proves\phi\vee(-\phi)=2(\phi\vee0)-\phi=2\phi^+-\phi$.\qedhere
\end{itemize}
\end{proof}

\begin{void}[Deduction Theorem]\label{prop_ded}
The following are equivalent:
\begin{itemize}
\item[1.] $T,\,\theta \proves\phi\lle\psi$;
\item[2.] $T\proves\phi-r\theta^-\lle\psi$ for some $r\in\QQ^+$.
\end{itemize}
\end{void}

\begin{proof}
By \ssf{4} of Proposition~\ref{prop_basic_inferences1}, we have $\theta\proves\theta^-=0$. So implication \ssf{2}$\IMP$\ssf{1} is clear. To prove \ssf{1}$\IMP$\ssf{2}, assume \ssf{1} and argue by induction on the length of a derivation of $\phi\le\psi$ from $T,\theta$. If the length is $1$, then either $\phi\lle\psi$ is in $T$ or it is an axiom,  in which case the conclusion holds by taking $r=0$, or $\phi\lle\psi$ syntactically coincide with $0\lle\theta$ and all we need to prove is $T\proves0-r\theta^-\lle\theta$ for some $r\in\QQ^+$. But this clearly holds by taking $r=1$ because $T\proves\theta\leq\theta^+-\theta^-$, by Proposition~\ref{prop_basic_Riesz}.

If the last applied rule in the derivation is either \ssf{r1} or \ssf{r2}, then proceed as in the proof of Theorem~\ref{prop_lin_ded}.

If the last applied rule is \ssf{r3} then $\phi$ and $\psi$ are of the form $\phi'\wedge0$ and $\psi'\wedge0$ respectively, for some $\phi',\psi'$ such that $T,\theta \proves\phi'\lle\psi'$. By induction hypothesis there exists $r\in\QQ^+$ such that $T\proves\phi'-r\theta^-\lle\psi'$. By applying rule \ssf{r3} we obtain $T\proves(\phi'-r\theta^-)\wedge0\lle\psi'\wedge0$. Now observe that $\vdash (\phi'-r\theta^-)\wedge(-r\theta^-)\le(\phi'-r\theta^-)\wedge 0$ and that $(\phi'-r\theta^-)\wedge(-r\theta^-)=(\phi'\wedge 0)-r\theta^-$ is an instance of axiom \ssf{a12}. Then
$T\proves(\phi'\wedge0)-r\theta^-\lle(\phi'-r\theta^-)\wedge0$. The conclusion thus follows.
\end{proof}

Next we  prove Proposition~\ref{lin_cut} for $\proves$.

\begin{proposition}\label{prop_cut_elim} 
The following are equivalent:
\begin{itemize}
\item[1.] $T\proves \psi$;
\item[2.] $T,\,\phi\proves \psi$ and $T,-\phi\proves \psi$.
\end{itemize}
\end{proposition}
\begin{proof}

As \ssf{1}$\IMP$\ssf{2} is trivial, we prove \ssf{2}$\IMP$\ssf{1}. Notice that $\proves(-\phi)^-=\phi^+$ so, assuming \ssf{2}, from the Deduction Theorem we get
\begin{itemize}
 \item[] $T\proves -r\phi^-\lle\psi$ and $T\proves -s\phi^+\lle\psi$
\end{itemize}
for some $r,s\in\QQ^+$. Hence, by \ssf{3} of Proposition~\ref{prop_basic_inferences1},
\begin{itemize}
\item[] $T\proves -(r\phi^-\wedge s\phi^+)\lle\psi$
\end{itemize}
 The conclusion follows if we show that $\vdash r\phi^+\wedge s\phi^-=0$. This is clear if $0\le r,s\le 1$, in fact $\vdash 0\le r\phi^+ \wedge s\phi^-\le \phi^+ \wedge \phi^-=0$. The general case follows by using axioms \ssf{a13} and \ssf{a8}.
\end{proof}

So far all the results apply to the restricted case and to the extended one.  This is not the case with the next result.

\begin{proposition}\label{prop_ext_bound}
In the extended case, for every formula $\phi$ there is some integer $n$ such that $\proves -n\lle\phi\lle n$.
\end{proposition}
\begin{proof}
 Follows from \ssf{a15} by straightforward induction on formulas.
\end{proof}

A consequence of the proposition above is that $-1\provesL\phi$ for every formula $\phi$. Notice that there is no formula that plays the role of a contradiction in the restricted case.

\begin{remark}\label{rk_solid} In this remark, which is not relevant for the further technical developments,  we motivate our definition of theory.

In the theory of Boolean algebras, there is a well-known correspondence between homomorphisms and filters (equivalently: ideals). We recall that a similar correspondence holds for Riesz spaces.  If $f: E\rightarrow L$ is  a homomorphism of Riesz spaces, then  $F=f^{-1}[0]$ is a so called \textit{solid subspace} of $E$  and $f$ factors through the quotient epimorphism $\pi: E\to E/F$, where $E/F$ is the quotient Riesz space of $E$ with respect to $F$. More precisely, there exists a unique Riesz space monomorphism $\iota:E/F\to L$ such that  $f=\iota\circ\pi$. See~\cite{Fremlin}*{\S 14G} for details.

It is also well-known that the quotient of the set of formulas of classical propositional logic with respect to the  relation of provable equivalence is a Boolean algebra and that deductively closed set of formulas become filters in such algebra (the improper filter corresponding to any inconsistent set).

Also, in $\RR$-valued logic, it can be easily verified that the quotient $ R/\mathord\sim$ of the set $R$ of formulas of any fixed language with respect to the equivalence relation defined by $\phi\sim\psi\  \Leftrightarrow\  \proves\phi\leq\psi$ is a Riesz space, when equipped with the induced operations. 
So, in order to carry on with the similarities, one should define a theory in $\RR$-valued logic as a set  $T$ of formulas such that $T/\mathord\sim$ is a solid subspace. Actually, our definition of theory as a set of inequalities is essentially equivalent to the above. For it is easy to verify  that, if $T$ is a theory in our sense, then  the set $\{(\phi_\sim,\psi_\sim):\  T\vdash \phi\le\psi\}$ is  a preorder on $ R/\mathord\sim$ which 
\begin{itemize}
\item[1.] extends the  ordering on $R/\mathord\sim$;
\item[2.] is compatible with the Riesz space structure (in the sense expressed by deduction rules \ssf{r1}\,$\div$\,\ssf{r3} above).
\end{itemize}
Moreover, if  we replace $R/\mathord\sim$ with an arbitrary Riesz space $E$ and  $\sqsubseteq$ is a preorder on $E$ which  satisfies \ssf{1} and \ssf{2} above, then the set  $F=\{v\in E:  0\sqsubseteq v \sqsubseteq 0 \}$ is a solid subspace.  Vice versa, if $F$ is a solid subspace of $E$,  by letting  $v \sqsubseteq w\ \Leftrightarrow\ v\le w+c$ for some $c\in F,$ we obtain a  preorder that  satisfies \ssf{1} and \ssf{2}.

To sum up: the definition of  theory  is just a matter of taste. For us,  a theory is essentially a partial  ordering on the set of formulas,   which can be interpreted by  saying that some formula is always truer than some other formula,  without  reference to a notion of  \textit{absolute truth}. On the contrary, in the alternative setting presented above,  the logical constant $0$ represents absolute truth, just as in Continuous Logic, and a theory contains all the ``absolutely  true''  formulas.
\end{remark}

\section{Linear formulas}\label{lin_for}

We say that formula $\phi$ is \emph{linear combination\/} of formulas $\phi_1,\dots,\phi_n$ if it is of the form

\textit{Restricted case\/}:\quad$\displaystyle\sum^n_{i=1} a_i\phi_i$\hfil\textit{Extended case\/}: \quad $\displaystyle\sum^n_{i=1} a_i\phi_i+a$\quad

for some $a_1,\dots,a_n,a\in\QQ$, with the convention that, if $n=0$, the two summations stand for $0$  and $a$ respectively. When $a_1,\dots,a_n,a\in\QQ^+$ we say that $\phi$ is a \emph{positive linear combination}. We say that $\phi$ is a \emph{linear formula\/} if it is a linear combination of pairwise distinct proposition letters. We call the corresponding tuple $(a_1,\dots,a_n)$ the \emph{vector associated\/} to $\phi$ and $a$ the \emph{affine component\/} of $\phi$. 

Notice that each formula free from the connective $\wedge$ is  $\provesL$-equivalent to an essentially unique linear formula, which, with slight abuse, we may call its normal form. In the sequel we shall always assume that linear formulas are in normal form.

For the reader's convenience, we provide an \textit{affine} version of Farkas' Lemma that is suitable for purposes, in the sense  that it is  formulated  with respect to the field of rationals. See the comment at the beginning of \cite{Sch}*{Ch. 7} about its  validity in the setting of rationals. See  \cite{Sch}*{Corollary 7.1h} for a proof.
 
\begin{void}[Farkas' Lemma]\label{Farkas}
Let  $v_1,\dots,v_n,u\in\QQ^k$ and let  $r_1,\dots,r_n,s\in\QQ$. Let $S=\{x\in\QQ^k : 0\le v_i\cdot x+r_i \textrm{ for } i=1,\dots, n\}$ be  non-empty.
Then the following are equivalent:
\begin{itemize}
\item[1.]  $0\le u\cdot x+ s$ for all $x\in S$;
\item[2.] there exist $q_1,\dots,q_n\in\QQ^+$ such that $\displaystyle u=\sum^n_{i=1} q_iv_i$ and  $\displaystyle \sum^n_{i=1} q_ir_i\le s$.
\end{itemize}
\end{void}

We also recall the following related result (see, for instance, \cite{Bar}*{Lemma 4.2} or \cite{Tao}*{Lemma 2.54}).

\begin{lemma}\label{duality}  Let $S$ be as in  Lemma~\ref{Farkas}. If $S$ is empty then there exist $q_1,\dots,q_n\in\QQ^+$ such that \smash{$\displaystyle 0=\sum q_iv_i$} and \smash{$\displaystyle \sum q_ir_i<0$.}
\end{lemma}

\begin{void}[Finite Strong Completeness Theorem for Linear Formulas]\label{prop_lin_completeness}
Let $\phi_1,\dots,\phi_n$ and $\psi$ be linear formulas. Then the following are equivalent:

\begin{itemize}
\item[1.] $\phi_1,\dots,\phi_n\provesL\psi$; 
\item[2.] $\phi_1,\dots,\phi_n\proves\psi$;
\item[3.] $\phi_1,\dots,\phi_n\models\psi$;
\item[4.] In the restricted case, $\psi$ is a positive linear combination of $\phi_1,\dots,\phi_n$. In the extended case, $\psi$ or $-1$  are positive linear combinations of $\phi_1,\dots,\phi_n$.
\end{itemize}
\end{void} 

\begin{proof}
Implication \ssf{4}$\IMP$\ssf{1} holds in both cases, under the assumption that $\psi$ is a positive linear combination of $\phi_1,\dots,\phi_n.$  It suffices to notice  that, for all formulas $\xi, \zeta$ and all $r\in\QQ^+,$ $\zeta \proves r\zeta$  and that $\zeta\proves \xi\le \xi+\zeta$  (both by rule \ssf{r2}).  Hence  $\xi, \xi\le \xi+\zeta \proves \xi+\zeta,$ by rule \ssf{r1}. Therefore $\xi, \zeta \proves \xi+\zeta$.  In the extended case, if $-1$ is a positive linear combination of $\phi_1,\dots,\phi_n$, then $\phi_1,\dots,\phi_n \provesL-1$ and, by the remark after Proposition~\ref{prop_ext_bound}, we get $\phi_1,\dots,\phi_n \provesL \psi$.

Implication \ssf{1}$\IMP$\ssf{2} is trivial and \ssf{2}$\IMP$\ssf{3} holds by soundness. Only \ssf{3}$\IMP$\ssf{4} is left to prove. Let $P_1,\dots,P_k$ be the proposition letters that occur in the formulas $\phi_1,\dots,\phi_n,\psi$. Let $v_1,\dots, v_n, u\in\QQ^k$ be the vectors associated to $\phi_1,\dots,\phi_n,\psi$ and let $r_1,\dots,r_n,s$ be their affine components (in the restricted case assume these to be $0$). 

In order to simultaneously deal with both  cases, in the extended case, without loss of generality we assume that the inequalities $-1\lle P_i\lle 1$, for $i=1,\dots, k$, occur among $\phi_1,\dots,\phi_n$. So \ssf{3} implies that
\begin{itemize}
\item[5.] for all $x\in\QQ^k$ if $0\le v_i\cdot x+r_i$ for  all $i=1,\dots,n$ then $0\le u\cdot x+s$.
\end{itemize}

Let $S$ be as in Lemma~\ref{Farkas}. Assume assume first that $S$ is non-empty (which is certainly true in the restricted case). By Lemma~\ref{Farkas} we get $q_1,\dots,q_n, r\in\QQ^+$  such that

\hfil$q_1v_1+\dots+q_n v_n=u\quad\mbox{and}\quad q_1r_1+\dots+ q_nr_n+r = s$.

It thus follows that $\psi$ is a positive linear combination of $\phi_1,\dots,\phi_n$. 

Eventually, we consider the case when $S$ is empty.  By Lemma~\ref{duality}, there exist $q_1,\dots,q_n\in\QQ^+$ such that $\provesL \ q_1\phi_1+\dots+q_n\phi_n=-1$, as claimed in \ssf{4}.
\end{proof}

\section{Finite strong completeness}\label{completeness}

Before proving a  strong completeness theorem for finite theories, we need a preliminary result.

Let $\bar\phi=(\phi_1,\dots,\phi_n)$ be a tuple of formulas and let $\epsilon=(\epsilon_1,\dots,\epsilon_n)\in\{-1,1\}^n$. We write $\epsilon\bar\phi$ for the tuple $\epsilon_1\phi_1,\dots,\epsilon_n\phi_n$.

\begin{lemma}\label{lem_spicchi}  For every formula  $\psi$  there exist a natural number $n_\psi$ and a tuple $\bar\psi$ of linear formulas of length $n_\psi$ with the property that, for each $\epsilon\in\{-1,1\}^{n_\psi}$ there exists a linear formula $\psi_\epsilon$  for which
\begin{itemize}
\item[1.]$\epsilon\bar\psi\vdash \psi =\psi_\epsilon$
\end{itemize}
\end{lemma}

\begin{proof} By induction on $\psi$.  The only non-trivial case is when $\psi$ is of the form $\varphi\wedge\xi$. Let us inductively assume the statement true for $\phi$ and $\xi$. Then, for all 
$\epsilon\in\{-1,1\}^{n_\phi}$ and all $\delta\in\{-1,1\}^{n_\xi},$

\begin{itemize}
\item[2.]$\epsilon\bar\varphi,\ \delta\bar\xi,\ \ \varphi_\epsilon - \xi_\delta 
\vdash \psi= \xi_\delta$
\item[3.] $\epsilon\bar\varphi,\ \delta\bar\xi, -\varphi_\epsilon + \xi_\delta \vdash \psi= 
\varphi_\epsilon$
\end{itemize}

Let $\bar\psi$ be the concatenation  of $\bar\phi$, $\bar\xi$ and $(\phi_\epsilon - \xi_\delta)_{\epsilon\delta}$. where the tuples $\epsilon\delta$ are lexicographically ordered. We denote by $p(\epsilon\delta)$ the position of $\epsilon\delta$ in such ordering.

Let $n_\psi=n_\phi+n_\xi+2^{n_\varphi+n_\xi}$. Notice that every $\sigma\in\{-1,1\}^{n_\psi}$ can be uniquely written as a concatenation $\epsilon\delta\rho$, for some $\epsilon\in\{-1,1\}^{n_\varphi}$, $\delta\in\{-1,1\}^{n_\xi}$  and $\rho\in\{-1,1\}^{2^{n_\varphi+n_\xi}}$. For each such $\sigma$ we let $\psi_\sigma=\xi_\delta $ if the $p(\epsilon\delta)$-th  coordinate of $\rho$ is 1 and $\psi_\sigma=\varphi_\epsilon$ if the $p(\epsilon\delta)$-th coordinate of $\rho$ is $-1$. It is now straightforward to check that $\sigma\bar\psi\vdash \psi=\psi_\sigma$.
\end{proof}

\begin{proposition}
Let $\phi$ be a formula. Then the following are equivalent:\nobreak
\begin{itemize}
\item[1.] $\proves\phi$;
\item[2.] $\models\phi$.
\end{itemize}
\end{proposition}
\begin{proof}
Direction \ssf{1}$\IMP$\ssf{2} is soundness. We prove \ssf{2}$\IMP$\ssf{1}.
Let $n_\phi,$  $\bar\phi$ and $\phi_\epsilon,$ for $\epsilon\in\{-1,1\}^{n_\phi},$ be as in the statement of Lemma~\ref{lem_spicchi}. Then
\begin{itemize}
\item[3.] $\epsilon\bar\phi\ \proves\ \phi=\phi_\epsilon$.
\end{itemize}

By soundness we also have $\epsilon\bar\phi\models\,\phi=\phi_\epsilon$. So, assuming $\models\phi$, we obtain that $\epsilon\bar\phi\models\,\phi_\epsilon$ for every $\epsilon\in\{-1,1\}^n$. From the finite strong completeness theorem for linear formulas, Theorem~\ref{prop_lin_completeness}, we  get  $\epsilon\bar\phi\provesL \phi_\epsilon$. Hence $\epsilon\bar\phi\proves\phi$ follows from \ssf{3}. Finally, we obtain \ssf{1} by repeatedly applying Proposition~\ref{prop_cut_elim}.
\end{proof}

\begin{void}[Finite Strong Completeness Theorem]\label{coroll_completeness_qf}%
Let $\phi$ be a formula and let $T$ be a finite theory. Then the following are equivalent:\nobreak%
\begin{itemize}
\item[1.] $T\proves\phi$;
\item[2.] $T\models\phi$.\qed
\end{itemize}
\end{void}

\bigskip
For infinite $T$ the previous corollary may fail. Consider the theory $T=\{0\lle rQ\lle P :r\in\QQ^+\}$ in the language $L=\{P,Q\}$. Then $T\models -Q$ but $T\notproves-Q$, otherwise we would have $T_0\proves-Q$ for some finite $T_0\subseteq T$. But $T_0\notmodels-Q$ for any finite $T_0\subseteq T$. We shall elaborate on this in the next section.

\begin{remark}\label{compl_char_mod}  With reference to the restricted case,  let us generalize  the definition of structure given in Section~\ref{structures}. For $R$ a Riesz space, a $R$-structure is a function $M: L\rightarrow R$. Truth in a $R$-structure is defined in the natural way and it is easy to verify that the corresponding semantics is sound for $\RR$-valued logic. Let us write $\models_{\mathcal R}$ for the  relation of logical consequence with respect to the class of all $R$-structures as $R$ ranges over all Riesz spaces. Then, for all formulas $\varphi$ and all finite theories $T$,  the following are equivalent:
\begin{enumerate}
\item[1.] $T\proves\phi;$
\item[2.] $T\models_{\mathcal R}\phi;$
\item[3.] $T\models\phi.$
\end{enumerate}

By Theorem~\ref{coroll_completeness_qf}, it suffices to show that \ssf{1} $\Rightarrow$  \ssf{2}. This implication is straightforward because Theorem~\ref{prop_soundness} easily extends to $\models_{\mathcal R}.$

Notice that the equivalences above yield the analogue of \cite[Theorem 2.12]{MO} for Riesz spaces.

\end{remark}

\begin{remark}
If we work with a countable language $L$ then $\RR$-valued logic is  decidable. Actually, the set $\{\varphi:\ \vdash\varphi\}$ is clearly recursively enumerable. Moreover, by the Completeness Theorem above we have that $\not\vdash\varphi$ if and only if there is some structure $M$ such that  $\varphi^M<0$. The latter holds if and only $\varphi^M<0$ for some $M: L\rightarrow{\mathbb Q}$. Let $P_1,\dots, P_n$ be the proposition letters occurring in $\phi$. Any effective enumeration of $\QQ^n$ induces an enumeration $(M_k)_{k\in\NN}$ of all rational valued assignments of values to $P_1,\dots, P_n$. The procedure that, at step $k$, computes $\varphi^{M_k}$ yields recursive enumerability of $\{\varphi: \ \not\vdash\varphi\}$.
\end{remark}

\begin{remark}\label{rmk_LvsRR}
Extended $\RR$-valued logic faithfully interprets  \L ukas\-iewicz logic.  In order to see that,  recall  that the  connectives $\dotminus$,  $\neg$  form a complete set of connectives for \L ukasiewicz logic and notice that they are definable in the extended logic. Moreover,  the models of  \L ukasiewicz logic form an axiomatizable subclass of those of extended $\RR$-valued logic, as  simply one has to impose the condition $0\lle P$ on every proposition letter $P$. From Theorem~\ref{coroll_completeness_qf} we get that, for every formula $\phi(P_1,\dots,P_n)$ in the language of \L ukasiewicz logic, whose proposition letters are among those displayed,

\hfil$\vdash_{\mbox{\scriptsize \L}} \phi(P_1,\dots,P_n)\quad \Leftrightarrow\quad  P_1,\dots, P_n\ \vdash\ \phi(P_1,\dots,P_n)=1$,

where $\vdash_{\mbox{\scriptsize \L}}$ and $\vdash$ stand for provability in  \L ukasiewicz and in the extended $\RR$-valued logic respectively.

Similar considerations apply to continuous propositional logic, recalling that  $\{\dotminus, \neg, \frac{1}{2}\}$  is a complete set of connectives for such logic.
\end{remark}

\section{Archimedean theories}\label{Archimedean}
 
As usual, we say that a theory $T$ is \emph{satisfiable\/} if there exists a structure $M$ such that $M\models T$. In the restricted case satisfiability property  is trivial,  as every theory is satisfiable in the constant model $0.$ Hence, in the restricted case, we are actually interested in satisfiability other than in the constant model $0$.

We say that  $T$ is \emph{consistent\/} if $T\notproves\phi$ for some formula $\phi$. 

We say that $T$ is \emph{total\/} if it is consistent and, for every formula $\phi$,  $T\proves\phi$ or $T\proves-\phi$, possibly both. In other words $T$ is total  if its deductive closure is a non-trivial total  preorder on the set of formulas. The reader may verify that $T$ is total if and only if it is \emph{prime}, namely if and only if it satisfies the property that whenever $T\proves\phi\vee\psi$ then $T\proves\phi$ or $T\proves\psi$.

\begin{proposition}
Every consistent theory extends to a total theory.
\end{proposition}
\begin{proof}
Let $\xi$ be a formula such that such that $T\notproves\xi$. Then  $T$ extends to a theory $T'$ that is maximal with respect to the property that $T'\notproves\xi$. Proposition~\ref{prop_cut_elim} implies that  $T'$ is total.
\end{proof}

We say that  $T$ is \emph{unital\/} if there is a formula $\xi$ such that, for every $\phi$, there is some $r\in\QQ^+$ such that $T\proves\phi\lle r\xi$. Such a formula $\xi$ is called a \emph{unit\/} in $T$. In the extended case, constant $1$ is a unit in all theories, as a consequence of Proposition~\ref{prop_ext_bound}.

Notice that if $\xi$ is a unit in $T$, then $\xi$ is a unit in every extension of $T$.

\begin{remark}\label{max_cons_ext} It is straightforward to check that, if $\xi$ is a unit in $T$, then $T$ is consistent if and only if $T\notproves-\xi$. Moreover, if $T$ is  maximal with respect to the latter property, then $T$ is maximal consistent. Therefore unital consistent theories extend to maximal consistent theories. We shall see this is not true in general.
\end{remark}

\begin{remark} Let $T$ be a theory and let $\phi$ be a formula such that $T\notproves\phi$. Then, by Proposition~\ref{prop_cut_elim}, $T, -\phi$ is consistent. We shall repeatedly use this this fact in the sequel without further mention.
\end{remark}

By the previous remark, every maximal consistent theory is total. In the extended case, $\Th(M)$ is trivially total for every model $M$.  The same holds in the restricted case, when $M$ is not the constant model  $0$. By Proposition~\ref{prop_max_con=tot+Arch} below, $\Th(M)$ is also maximal consistent.

We write \emph{$T^+$\/} for the set of formulas $\phi$ such that $T\proves\phi$ and \emph{$T^-$} for the set of formulas $\phi$ such that  $T\proves -\phi$. If $T^+= T^-$  we say that $T$ is \emph{trivial}. So, if $T$ is non-trivial then $T^+\sm T^-$ is non-empty. Non-trivial theories are consistent: just notice that if $\xi\in T^+\sm T^-$ then $T\notproves-\xi$. The converse is not true in general, the empty theory being a counter-example. Complete theories are easily seen to be non-trivial. Consistent unital theories are non-trivial as well. 

We write $\QQ^+\phi$ for the set $\{r\phi:r\in\QQ^+\}$. 

\begin{proposition}\label{prop_maxcon_vs_comp}
Let $T$ be a consistent theory and let $\phi,\psi$ be such that $T\proves\QQ^+\phi\le\psi$. Then $T,-\phi$  is consistent.
\end{proposition}
\begin{proof}
If $T, -\varphi$ is inconsistent then $T, -\varphi \vdash  -\psi$. By assumption and by the Deduction Theorem, we get  $T\vdash \QQ^+\varphi \le r\varphi^+,$ for  some $r\in{\mathbb Q}^+$.  From inconsistency of  $T, -\varphi$ we also get $T\vdash\varphi$ hence $T\vdash\varphi=\phi^+$. Therefore $T\vdash -\varphi$, contradicting the consistency of $T$.
\end{proof}

We say that $T$ is \emph{Archimedean} if for every $\phi,\psi$ such that $T\proves\QQ^+\phi\le\psi$ then $T\proves-\phi$. Clearly $\Th(M)$ is Archimedean for every structure $M$. It follows from Theorem~\ref{coroll_completeness_qf} that every finite theory is Archimedean.

\begin{proposition}\label{prop_max_con=tot+Arch}
For every theory $T$, the following are equivalent:
\begin{itemize}
\item[1.] $T$ is maximal consistent;
\item[2.] $T$ is closed under deduction, total and Archimedean;
\item[3.] $T$ is closed under deduction, total and every $\xi\in T^+\sm T^-$ is a unit in $T$.
\end{itemize}
\end{proposition}

\begin{proof} Implication \ssf{1}$\IMP$\ssf{2} is straightforward. As for \ssf{2}$\IMP$\ssf{3}, assume \ssf{2} and, for sake of contradiction,  let $\xi\in T^+\sm T^-$ which is not a unit. Then there is a formula $\phi$ such that $T\notproves \phi\lle r\xi$ for any $r\in\QQ^+$. Being total,  $T\proves r\xi\lle\phi$ for all $r\in\QQ^+$. Hence, by the Archimedean property, $T\proves -\xi$, a contradiction. 

As for \ssf{3}$\IMP$\ssf{1}, assume \ssf{3} and let $\xi\notin T$. Completeness and closure under deduction yield that $-\xi\in T$. Hence, by assumption, $-\xi$ is a unit in $T$. Therefore for every $\phi$ there exists $r\in\QQ^+$ such that $T\proves r\xi\lle-\phi$. So $T,\xi$ is inconsistent.
\end{proof}

\begin{example}\label{ex_non_Archiemedeo}
There is a total theory $T$ that has no maximal consistent extension. In particular, $T$ is non-Archimedean, non-unital, and has only the constant model $0$.

We construct such a theory $T$ by using the generalized notion of structure introduced in Remark~\ref{compl_char_mod}. We are going to define a $\QQ[x]$-valued structure $M$. Recall that  $\QQ[x]$ has an ordered ring structure with respect to  the order induced by $r<x$, for all $r\in\QQ$. For each $i\in\omega$, the language of $M$ contains a proposition letter $P_i$ which is interpreted as the monomial $x^i$. Let $T$ be the set of inequalities that hold in $M$. As the generalized semantics  is sound, the theory $T$ is consistent and deductively closed. As the ordering on $\QQ[x]$ is a linear, $T$ is total. Moreover $T\proves 0\lle \QQ^+ P_i\lle P_{i+1}$ for all $i$ so, if  $S$ is any Archimedean  extension of $T$, then $S\models P_i=0$.  It follows that   $S$ is inconsistent.
We have just shown that  $T$ has no consistent Archimedean extension, so it has only the constant model $0$.\qed
\end{example}

A comment on the generalized semantics above: just notice that every inequality (in our sense) can be regarded  as an inequality in the language of Riesz spaces (metavariables for formulas becoming variables for vectors).  Keeping this in mind, the axioms of  $\RR$-valued logic  are true in every Riesz space and its deduction rules can be regarded as  valid deduction rules in the Riesz space setting. So, if $N$ is any Riesz space and $V\subseteq N$ is a set of linearly independent vectors, any set $T$ of weak inequalities satisfied in $N$ by  linear 
combinations of elements of $V\cup\{0\}$
can be viewed as a consistent theory in our sense, simply by regarding each
$v\in V$ as a proposition letter. In particular, when  $N$ is linearly ordered,  the set  $T$ of all 
inequalities
satisfied in $N$ by  linear combinations of  elements of $V\cup\{0\}$, 
we get  a total theory.

\begin{void}[Weak Completeness Theorem for Unital Theories]\label{prop_cons_has_model}Let $T$ be a unital theory and let $\xi$ be a unit in $T$. In the extended case further assume that  $\xi$ is $1$. Then the following are equivalent:
\begin{itemize}
\item[1.]  $T$ is  consistent;
\item[2.]  there exists a structure $M$ such that $M\models T$ and  $\xi^M=1$.
\end{itemize}
\end{void} 
\begin{proof}
Implication \ssf{2}$\IMP$\ssf{1} follows immediately from soundness. To prove \ssf{1}$\IMP$\ssf{2}, let $S$ be a maximal consistent extension of $T$ (see Remark~\ref{max_cons_ext}). Notice that, by Proposition~\ref{prop_max_con=tot+Arch}, $S$ is total and Archimedean. If $\zeta$ is a formula, we let 

\hfil$S_\zeta\ =\ \big\{r\in\QQ\ :\  S\proves \zeta\lle r\xi\big\}\quad\mbox{and}\quad S^\zeta\ =\ \big\{r\in\QQ\ :\ S\proves r\xi\lle\zeta\big\}$.

Totality and Archimedean property of $S$ imply that, for every formula $\zeta,$
\begin{itemize}
\item[1.] $\emptyset\ne S^\zeta\le S_\zeta\ne\emptyset$;  
\item[2.] $S_\zeta$ is bounded from below;  
\item[3.] $\inf S_\zeta=\sup S^\zeta$. 
\end{itemize} 

We define a structure $M$ as follows: for every proposition letter $P$  we let
\begin{itemize}
\item[4.]  $P^M\ =\ \inf S_P\ =\ \ \sup  S^P$.
\end{itemize}

Recall that $0^M=0$ and, in the extended case, $1^M=1$. Notice that, by consistency of $S$, the equalities in \ssf{3} hold even when $P$ is replaced by $0$ and, in the extended case, also by constant  $1$. We prove by  induction  that \ssf{4} extends to all  formulas. 

Let $\zeta$ be  of the form $\phi+\psi$. Let us inductively assume that $\phi^M=\inf S_\phi=\sup S^\phi$ and $\psi^M=\inf S_\psi=\sup S^\psi$. Notice that $S_\phi+S_\psi\subseteq S_{\phi+\psi}$ and that $S^\phi+S^\psi\subseteq S^{\phi+\psi}$. Then 

$$\begin{array}{ccl}
\phi^M+\psi^M & =  & \sup S^\phi+\sup S^\psi\  =\  \sup(S^\phi+S^\psi)\\[2mm]
& \le & \sup S^{\phi+\psi}\  =\  \inf  S_{\phi+\psi}\\[2mm]
&\le  & \inf (S_\phi+S_\psi)\ =\ \inf S_\phi + \inf S_\psi\ =\ \phi^M+\psi^M.
\end{array}
$$

As $\zeta^M=\phi^M+\psi^M$, we are done.

Next, let $\zeta$ be of the form $\phi\land\psi$.  Notice that   $S_\phi, S_\psi\subseteq S_{\phi\land\psi}.$ Without loss of generality, assume $\phi^M\le\psi^M$. Under the same inductive assumptions as in the previous case, we get
$\phi^M=\inf S_\phi\ge\inf S_{\phi\land\psi}.$  We claim that the previous inequality cannot be strict. For sake of contradiction, suppose it is and let $r,s\in\QQ$ be such that $\inf S_{\phi\land\psi}<r<s<\inf S_\phi.$ By totality of $S$ and by $\phi^M\le\psi^M$ we get $S\proves s\xi\le\phi$ and $S\proves s\xi\le\psi.$ Hence, by \ssf{3} of Proposition~\ref{prop_basic_inferences1}, $S\proves s\xi\le \phi\land\psi.$ On the other hand, $S\proves \phi\land\psi\le r\xi.$ It follows that $S\proves (r-s)\xi,$ contradicting consistency of $S.$ 
Therefore $(\phi\land\psi)^M=\inf S_\phi=\inf S_{\phi\land\psi}=\sup S^{\phi\land\psi}.$

The case when $\zeta$ is of the form $q\phi$ is straightforward.

Finally, it is easy to check that $M\models S$ and $\xi^M=1$.
\end{proof}

Among unital theories, Archimedean theories are those for which a strong completeness theorem holds.

\begin{void}[Completeness Theorem for Archimedean Unital Theories]\label{compl_uni} 
Let $T$ be a consistent unital theory. Then the following are equivalent:
\begin{itemize}
\item[1.]  $T$ is Archimedean;
\item[2.]  for every formula $\phi$, if $T\models\phi$ then  $T\proves\phi$. 
\end{itemize}
Moreover \ssf{2}$\IMP$\ssf{1} also holds for non-unital theories.
\end{void}

\begin{proof}
To prove \ssf{2}$\IMP$\ssf{1}, assume \ssf{2} and suppose that  $T\proves\QQ^+\phi\lle\psi$ and $T\notproves-\phi$,  for some $\phi,\psi$. By \ssf{2}, there  exists $M\models T$ such that $\phi^M>0$, which immediately yields a contradiction.

To prove \ssf{1}$\IMP$\ssf{2},  suppose that $T$ is Archimedean. Let $\phi$ be such that $T\notproves-\phi$ and  let  $\xi$  be a unit in $T$.  In the extended case further assume that $\xi$ is  $1$. Let $r\in\QQ^+$ be such that $T\notproves r\phi\lle\xi$. Then $T, \xi\lle r\phi$ is unital and  consistent. By Theorem~\ref{prop_cons_has_model} there exists $M\models T,\xi\le r\phi$ such that $\xi^M=1$. From  $M\models r^{-1}\xi\lle\phi$ we get $T\not\models -\phi$.
\end{proof}

An approximated completeness theorem, similar to that proved in \cite{BY} for  continuous logic, holds for unital theories. It is an immediate corollary of Theorem~\ref{prop_cons_has_model}.

\begin{void}[Approximated Completeness Theorem for unital theories]
Let $T$ be a consistent unital theory. Then the following are equivalent for every unit $\xi$:
\begin{itemize}
\item[1.]  $T\proves \phi+r\xi$ for all $0<r\in\QQ$;
\item[2.]  $T\models\phi$.
\end{itemize}
\end{void}

\begin{proof}
Implication \ssf{1}$\IMP$\ssf{2} follows from soundness. To prove \ssf{2}$\IMP$\ssf{1}, let $\xi$ be a unit in $T$.  In the extended case further assume that $\xi$ is  $1$. Assume that $T\notproves \phi+r\xi$, for some $0<r\in\QQ$. Hence $T,-(\phi+r\xi)$ is consistent and unital. By Theorem~\ref{prop_cons_has_model}, there exists $M\models T$ such that $M\models \phi+r\xi\lle0$ and $\xi^M=1$. Such $M$ witnesses $T\notmodels\phi$.
\end{proof}

{\sc Stefano Baratella}\\
Dipartimento di Matematica\\
Universit\`a di Trento\\
via Sommarive 14, Povo, 38123 Trento\\
{\tt\small stefano.baratella@unitn.it}

\bigskip
{\sc Domenico Zambella}\\
Dipartimento di Matematica\\
Universit\`a di Torino\\
via Carlo Alberto 10, 10123 Torino\\
{\tt\small domenico.zambella@unito.it}

\end{document}